\documentclass[12pt]{article}
\usepackage{latexsym,color,amsmath,amsthm,amssymb,amscd,amsfonts}
\usepackage{scalefnt}
\usepackage[font=small,labelfont=bf]{caption}
\usepackage{float}
\usepackage{graphicx}
\usepackage{amsopn}
\usepackage{relsize}
\usepackage{mathrsfs}
\usepackage{upgreek}

\newtheoremstyle{nonum}{}{}{\itshape}{}{\bfseries}{.}{ }{\thmnote{#3}}

\newtheorem{thm}{Theorem}[section]
\newtheorem*{thm*}{Theorem}

\newtheorem{lem}[thm]{Lemma}

\newtheorem{conj}[thm]{Conjecture}

\newtheorem*{definition*}{Definition}

\newtheorem*{rems*}{Remarks}

\theoremstyle{nonum}

\newcommand{\R}{\mathbb R}
\newcommand{\RR}{\mathbb R}

\def\vol{{ \rm vol_n}}

\begin{document}
\title{Convex Floating Bodies of Equilibrium }
\date{}
\author{D.I. Florentin, C. Sch\"utt, E.M. Werner\thanks{Partially
		supported by NSF grant DMS-1811146 and by a Simons Fellowship}, N. Zhang}
\maketitle
\begin{abstract}
We study a long standing open problem by Ulam, which is whether the
Euclidean ball is the unique body of uniform density which will float
in equilibrium in any direction. We answer this problem in the class
of origin symmetric  $n$-dimensional convex bodies whose relative density  to  water
is $\frac{1}{2}$. For $n=3$, this result is due to Falconer.
\end{abstract}

\section{Introduction and results}\label{Sec_Intro}
\subsection{Ulam floating bodies}

A long standing open problem asked by Ulam in  \cite{Ulam2} (see also  \cite{Ulam1}, Problem 19), 
is whether the Euclidean ball is the unique body of uniform density
which floats in a liquid in equilibrium in any direction. We call
such a body {\it Ulam floating body}.
\par
\noindent
A two-dimensional counterexample was found for relative density
$\rho = \frac12$ by Auerbach \cite{Auerbach} and for densities
$\rho \neq \frac12$ by Wegner \cite{Wegner1}. These counterexamples
are not origin symmetric. For higher dimension, Wegner obtained 
results for non-convex bodies (holes in the body are allowed) in
\cite{Wegner2}. The problem remains largely open in the class of
convex bodies in higher dimension. In
order to study Ulam floating bodies, we use the notion of the {\it
convex floating body}, which was introduced independently by
B\'ar\'any and Larman \cite{BaranyLarman1988} and by Sch\"utt and
Werner \cite{SW1}. Let $K$ be a convex body in $\RR^n$ and let
$\delta \in \RR$, $0 \le \delta \leq \frac{1}{2}$. Then the convex
floating body $K_\delta$ is defined as
\[
K_\delta =
\bigcap_{u \in S^{n-1}}
H^-_{\delta,u}.
\]
Here $H^+_{\delta,u}$ is the halfspace with outer unit normal vector
$u$, which ``cuts off'' a  $\delta$  proportion of $K$,
i.e. $\vol\left( K\cap H^+_{\delta, u} \right) = \delta \  \vol(K)$.
The convex floating body is a natural variation of Dupin's floating
body $K_{[\delta]}$ (see \cite{Dupin}). $K_{[\delta]}$ is this convex body 
contained in $K$ such that every support hyperplane cuts off a set of 
volume $\delta \  \vol(K)$ exactly. 
In general $K_{[\delta]}$ need not exist.  An example is e.g.,  the simplex in $\R^n$.
Dupin showed that if it  exists, each supporting hyperplane $H$ touches $K_{[\delta]}$ at the centroid of $K \cap H$.
If the floating body $K_{[\delta]}$ exists, it is equal to the convex floating body $K_\delta$.
It was shown in \cite{MR} that for a
symmetric convex body $K$, one has $K_{[\delta]} = K_\delta$.
\par
We recall the relation between the density $\rho$ and the volume  $\delta\vol(K)$ that is cut off.
If the liquid has density $1$ and the body $K$ has unit volume and
density $\rho$, then by Archimedes' law the submerged volume equals
the total mass of the body, i.e. $\rho$, and consequently the floating
part has volume $\delta = 1 -\rho$.
\par
In \cite{HS}, the authors defined the {\em metronoid} $M_\delta(K)$
of a convex body $K$ to be the body whose boundary consists of
centroids of the floating parts of $K$, i.e. $K\cap H^+_{\delta, u}$.
More precisely, denoting
$X_{K,\delta}(u) =
\frac{1}{\delta \vol(K)} \int_{K\cap H_{\delta,u}^+} x\,dx$,
they defined $M_{\delta}(K)$ by
\[
\partial M_\delta(K) =
\left\{
X_{K,\delta}(u) \,:\, u\in S^{n-1}\right\},
\]
and showed that $X_{K,\delta} : S^{n-1} \to \partial M_\delta(K)$ is
the Gauss map of $M_\delta(K)$, i.e. the normal to $M_{\delta}(K)$ at
$X_{K,\delta}(u)$ is $u$. Huang, Slomka and  Werner showed that $K$
is an Ulam floating body if and only if $M_\delta(K)$ is a ball (see
\cite[Section 2.2]{HSW} for details). We utilize this
characterization in our proof of Theorem \ref{thm1}.

\subsection{Main results}
We present two results concerning Ulam's problem. We first establish
a relation between Ulam floating bodies and a uniform
isotropicity property of sections. 
Our second
result provides a short proof of a known answer to Ulam's problem 
in the class of symmetric convex bodies with relative density
$\rho = 1/2$.  
\par
\noindent
In the next theorem, and elsewhere, we use the notation $g(B)$ for the centroid of a set $B$.
\par
\noindent
\begin{thm}\label{thm1}
Let $\delta \in \left(0, \frac12\right]$ and let $K\subset\R^n$ be a convex body such
that $K_\delta$ is $C^1$ or $K_\delta = K_{[\delta]}$ reduces to a point.
Then $K$ is an Ulam floating body 
if and only if there exists $R>0$ such that for all $u \in S^{n-1}$
and $v \in S^{n-1} \cap u^\perp$,
\begin{equation}\label{eq-isotropic}
\int_{K\cap H_{\delta, u}} \langle x, v \rangle^2 - \langle  g(K \cap  H_{\delta, u}), v \rangle^2 \,dx = \delta  \  \vol(K) R.
\end{equation}
In that case, $M_\delta(K)$ is a ball of radius $R$.
\end{thm}
\par
\noindent
{\bf Remark.}
Note that if $K_\delta $ reduces to a point, which without loss of generality we can assume to be $0$,  then the condition (\ref {eq-isotropic}) reduces to 
\begin{equation}\label{eq-isotropic point}
\int_{K\cap H_{\delta, u}} \langle x, v \rangle^2  \,dx = \delta  \ \vol(K) R.
\end{equation}
\vskip 2mm
\noindent
We use the characterization given in Theorem \ref{thm1} to give
a short proof of the following result which was proved in dimension $3$ by Falconer \cite{Falconer}.
It also follows from a result in \cite{Schneider}.
\par
\noindent
\begin{thm}\label{thm2}
Let $K\subset\R^n$ be a symmetric convex body of  volume $1$ and density $\frac12$.
If $K$ is an Ulam floating body, then $K$ is a ball.
\end{thm}
\section{Background}
We collect some definitions and basic results that we use throughout the paper. For further facts in convex geometry we refer the reader  to the books by Gardner \cite{GardnerBook} and Schneider \cite{SchneiderBook}.
 \vskip 2mm
 \noindent
The radial function $r_{K, p}: S^{n-1} \to \R^+$ of a convex body $K$   about a point $p\in\R^n$ is defined by
\[
r_{K,p} (u) = \max\{\lambda \ge 0: \lambda u  \in K-p\}.
\]
If $0 \in \text{int}(K)$, the interior of $K$, we simply write $r_K$ instead of $r_{K,0}$.
\newline
Let $K\subset \R^n$ be a convex body
containing a strictly convex body $D$ in its interior, and let $H$ be
a hyperplane supporting $D$ at a point $p$. If $u$ is the outer unit
normal vector at $p$, we denote the restriction of the radial
function $r_{K\cap H, p}\,$ to $S^{n-1}\cap u^\perp\,$ by
$\,r_{K,D}(u, \cdot)$.
\par
\noindent
We denote by $B^n_2$ the Euclidean unit ball centered at $0$ and by $S^{n-1} = \partial B^n_2$ its boundary.
The spherical Radon transform $R:C(S^{n-1}) \to C(S^{n-1})$ is defined by 
\begin{equation} \label{Radon}
R f(u) = \int_{u^\perp \cap S^{n-1}} f(x) dx
\end{equation}
for every $f \in C(S^{n-1})$.
\vskip 2mm
\noindent
\subsection{Some results on floating bodies}

Since $\delta > \frac{1}{2}$ implies $K_\delta = \emptyset$, we restrict our
attention to the range $\delta \in \left[0, \frac{1}{2} \right]$. 
It was shown in \cite{SchuettWerner94} that there is   $\delta_c$, $0 < \delta_c \leq \frac{1}{2}$ such that $K_{\delta_c}$ 
reduces to a point. It can happen that $\delta_c < \frac {1}{2}$. An example is the simplex.
\par
\noindent
In fact, Helly's Theorem (and a simple union bound) implies that 
$\delta_c > \frac {1}{n+1}$, so we have $\delta_c \in
\left( \frac {1}{n+1}, \frac{1}{2} \right]$. 
\par
\noindent
As mentioned above, when  Dupin's floating body $K_{[\delta]}$
exists, it coincides with the convex floating body $K_\delta$.
The following lemma states that existence of $K_{[\delta]}$ is also
guaranteed by smoothness of $K_\delta$. We use this for Theorem \ref{thm1}.
\begin{lem}
	If $K_\delta$ is $C^1$, then $K_{[\delta]}$ exists and $K_\delta=K_{[\delta]}$.
\end{lem}
\par
\noindent
\begin{proof}
	Let  $x \in \partial K_\delta$. By \cite{SchuettWerner94},
	there is at least one hyperplane $H$ through $x$ that cuts off
	exactly $\delta \vol(K)$ from $K$ and this hyperplane touches
	$\partial K_\delta$ in the barycenter of $H\cap K$. As $K_\delta$ is
	$C^1$, the hyperplane at every boundary point $x \in K_\delta$ is
	unique. Thus $K_{[\delta]}$ exists and  $K_\delta = K_{[\delta]}$.
\end{proof}

\vskip 3mm
\noindent
We know that $K_{\frac {1}{2}} = \{0\}$ for every centrally symmetric convex body $K$.
The Homothety Conjecture  \cite{SchuettWerner94} (see also \cite{Stancu, WernerYe}),  states that the
homothety $K_\delta = t(\delta)  K$ only occurs for ellipsoids. We treat the
following  conjecture which has a similar flavor.
\begin{conj}
	Let $K\subset \R^n$ be a convex body and let $\delta\in \left( 0, \frac {1}{2}\right)$.
	If $K_\delta$ is a Euclidean ball, then $K$ is a Euclidean ball.
\end{conj}
\par
\noindent
We now prove the two dimensional case of the conjecture. 
\par
\noindent
\begin{thm}
	Let $K\subset \R^2$ and suppose there is  $\delta \in \left(0, \frac {1}{2}\right)$ such
	that $K_\delta = r \  B^2_2$. Then $K = R \  B^2_2$, for some $R>0$.
\end{thm}
\begin{proof}
	We shall prove that the radial function $r_K: S^1 \to \R$ is
	constant. If the continuous function $r_K$ is not constant, it
	must attain some value $R>r$ such that the angle $\theta =
	\arccos\left(\frac{r}{R}\right) \in \left(0,\frac{\pi}{2}\right)$ is
	not a rational multiple of $\pi$. Let $u_1\in S^1$ be the point with
	$r_K(u_1) = R$, and let $\left\{u_i \right\}_{i=1}^\infty$ be the
	arithmetic progression on $S^1$ with difference $2\theta$. We claim
	that $r_K$ is constant on $\left\{u_i \right\}$. Indeed, assuming
	$R u_i \in \partial K$, we consider the triangle with vertices
	$O,\, Ru_i, Ru_{i+1}$ (see the figure below). The edge
	$\left[Ru_i, Ru_{i+1}\right]$ is tangent to $K_\delta = r B^2_2$ at its
	midpoint $m_i$, and since $K_\delta$ is smooth, the chord on
	$\partial K$ containing $Ru_i$ and $m_i$ is bisected by $m_i$, which
	implies $Ru_{i+1} \in \partial K$, i.e., $r_K(u_{i+1}) = R$. Since
	$\theta$ is not a rational multiple of $\pi$, the sequence
	$\left\{u_i \right\}$ is dense in $S^1$. Since $r_K$ is constant
	on a dense set and continuous, it is constant on $S^1$, as required.
	\begin{center}
		\includegraphics[width=0.6\linewidth]{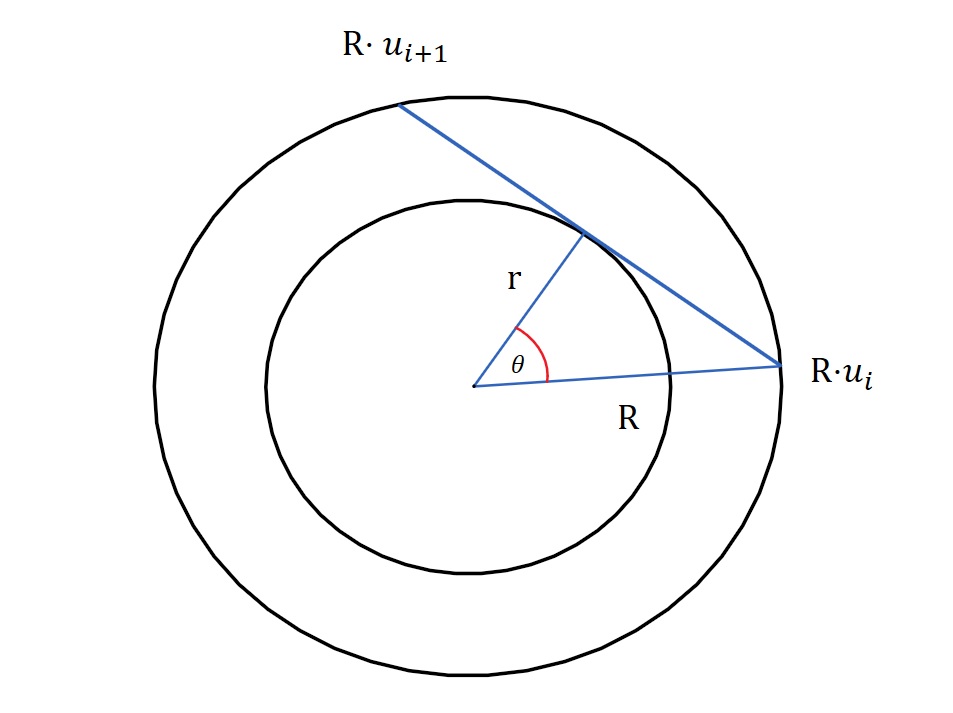}
	\end{center}
	
\end{proof}

\section{Proof of the main theorems}
\subsection{Proof of Theorem \ref{thm1}}

\begin{proof}
We first treat the case $n=2$. Also, we first consider when $K_\delta$ reduces to a point. Without loss of generality we can
assume that this point is $0$. Then we have for all $u \in S^1$ that $g(K \cap H_{\delta,u}) = 0$ by Dupin and thus  $\langle g(K \cap H_{\delta,u}), v \rangle =0$, for all
$v \in u^\perp \cap S^1$ and the condition reduces to $\int_{K \cap H_{\delta,u}}  \langle x, v \rangle^2 dx = C$. This observation is true in all dimensions.
\newline
Let $u \in S^1$. Let $v \in u^\perp \cap S^1 = H_{\delta,u} \cap S^1$. Then, as $H_{\delta,u}= \text{span}\{v\}$, we get for all $v \in S^1$,
\begin{eqnarray}\label{Gleichung1}
\int_{K \cap H_{\delta,u}}  \langle x, v \rangle^2 dx =
\int_{-r_K(v) }^{r_K(v)}x^2 dx = \frac{2}{3} r_K(v)^3,
\end{eqnarray}
as $r_K(v)=r_K(-v)$.
Hence if $\int_{K \cap H_{\delta,u}}  \langle x, v \rangle^2 dx=C$, then we get that for all $v \in S^1$ that $r_K(v)=C_1$, and hence 
$K$ is a Euclidean ball and therefore also $M_\delta(K)$ is a Euclidean ball.
\par
 For the other direction, we fix $u \in S^1$. We may assume that $u=(1,0)$, i.e., in polar coordinates
$u_0$ corresponds to $\theta=0$ and $r(\theta)=1$. For $\phi$ small, let $w=(\cos \phi, \sin\phi)$ and define the sets
\[
E_1 = H_{\delta, u}^+ \cap H_{\delta, w}^+ \cap K, \quad
E_2 = H_{\delta, u}^+ \cap H_{\delta, w}^- \cap K, \quad
E_3 = H_{\delta, u}^- \cap H_{\delta, w}^+ \cap K.
\]
In order to compute the derivative of  the boundary curve of $M_\delta(K)$ we
write
\begin{eqnarray*}
&&\delta \text{vol}_2(K) \cdot \big[ X_{K,\delta}(w) - X_{K,\delta}(u) \big]   \\
&&	= \int_{ E_1\cup  E_3}x\,dx - \int_{ E_1\cup  E_2}x\,dx  =   \int_{E_3}x\,dx - \int_{E_2}x\,dx\\
&&=\int_{\frac{\pi}{2}}^{\frac{\pi}{2} + \phi} 
  \int_{0}^{r_{K}(\theta)} 
	(\cos\theta, \sin \theta) \  r^2 \,dr\,d\theta
-
  \int_{-\frac{\pi}{2}}^{-\frac{\pi}{2} +\phi}
  \int_{0}^{r_{K}(\theta)}
	(\cos\theta, \sin \theta) \  r^2 \,dr\,d\theta  \\
&&=2 \int_{\frac{\pi}{2}}^{\frac{\pi}{2} +\phi} \int_{0}^{r_{K}(\theta)}
	(\cos\theta, \sin \theta) \  r^2 \,dr\,d\theta  = 
 \frac{2}{3} \int_{\frac{\pi}{2}}^{\frac{\pi}{2} +\phi}  r_K(\theta)^3 (\cos\theta, \sin \theta) \  d\theta.
\end{eqnarray*}
Thus
\begin{eqnarray*}
\frac{d}{ d \phi}\left[ X_{K,\delta}(w) - X_{K,\delta}(u) \right]
 =  \frac{2}{3 \  \delta \  \text{vol}_2(K) } \    r_K\left(\phi +\frac{\pi}{2} \right)^3 (- \sin \phi, \cos\phi)
\end{eqnarray*}
and hence
\begin{eqnarray}\label{Gleichung2}
\left|\frac{d}{ d \phi} \left[ X_{K,\delta}(w) - X_{K,\delta}(u) \right] \right|=  \frac{2}{3 \  \delta\   \text{vol}_2(K)} \    r_K\left(\phi +\frac{\pi}{2} \right)^3, 
\end{eqnarray}
where $| \cdot |$ denotes the Euclidean norm.
With  $z= w^\perp$, we get from (\ref{Gleichung1}) and (\ref{Gleichung2}) that 
\begin{eqnarray}\label{Gleichung3}
\left|\frac{d}{ d \phi} \left[  X_{K,\delta}(w) - X_{K,\delta}(u) \right] \right|=  \frac{1}{\delta \   \text{vol}_2(K) } \int_{K  \cap H_{\delta,w}} \langle x, z \rangle ^2 dx.
\end{eqnarray}
If $M_\delta(K)$ is a Euclidean ball with radius $R$, we write $X_{K,\delta}\   (\cos \phi, \sin\phi)= R \   (\cos \phi, \sin\phi)$ in polar coordinates 
and get from (\ref{Gleichung3})
\begin{eqnarray*}
\frac{1}{\delta \   \text{vol}_2(K) } \int_{K   \cap H_{\delta,w} } \langle x, z \rangle ^2 dx &=& \left|\frac{d}{ d \phi} \left[  X_{K,\delta}(w) - X_{K,\delta}(u) \right] \right|\\
&=&\left|\frac{d}{ d \phi} \left[   X_{K,\delta}(w) \right] \right| = R.
\end{eqnarray*}
\vskip 3mm
\noindent
To treat the case when $K_\delta$ does not consist of just one point, 
we introduce the following
coordinate system for the complement of an open, strictly convex body
$D\subset\R^2$ with smooth boundary (see also e.g., \cite{YaskinZhang}).   Let
$\gamma : [0,2\pi] \to \partial D$ be the inverse Gauss map, and
$T: [0,2\pi] \to S^1$ be the unit tangent vector to the curve at
$\gamma(\theta)$, oriented counterclockwise, i.e., 
\[
n(\theta):=n_D(\gamma(\theta)) =
\left(
\begin{array}{c}
\cos \theta \\
\sin \theta
\end{array}
\right),\qquad
T(\theta) =
\frac{\gamma'(\theta) }{\big| \gamma'(\theta) \big|} = 
\left(
\begin{array}{c}
-\sin \theta\\
\cos \theta 
\end{array}
\right).
\]
The coordinate system $F: \R\times [0, 2\pi] \to \R^2 \setminus D$ is
defined by
\begin{equation}\label{def-coord-sys}
F(r,\theta) = \gamma(\theta) + r T(\theta).
\end{equation}
Since $\frac{\partial F}{\partial r} = T$ and
$\frac{\partial F}{\partial \theta} = \gamma' -r n$, the Jacobian of
$F$ is given by $|r|$. 
\par
\noindent
Now we fix $0 < \delta <1/2$ and assume that $K_\delta$ does not just consist of one point. We then  set $D = \text{int} (K_\delta)$,  which has 
smooth boundary by assumption. 
Without loss of generality, we can assume that $0 \in \text{int} (K_\delta)$. 
It was shown in \cite{SchuettWerner94} that $K_\delta$ is strictly convex. 
Let $u\in S^1$, and assume without loss of
generality that $u = n(0)=(1, 0)$.
Let $w = n(\phi)$ for an angle $\phi>0$ small enough, such that
the lines $H_{\delta,u}$ and $H_{\delta,w}$ intersect in the interior
of $K$. Define the sets
\[
E_1 = H_{\delta, u}^+ \cap H_{\delta, w}^+ \cap K, \quad
E_2 = H_{\delta, u}^+ \cap H_{\delta, w}^- \cap K, \quad
E_3 = H_{\delta, u}^- \cap H_{\delta, w}^+ \cap K, 
\]
and let $E_4$ be the bounded connected component of
$\left(H_{\delta, u}^- \cap H_{\delta, w}^-\right) \setminus K_\delta$, 
see Figure.
\begin{center}
	\includegraphics[width=0.7\linewidth]{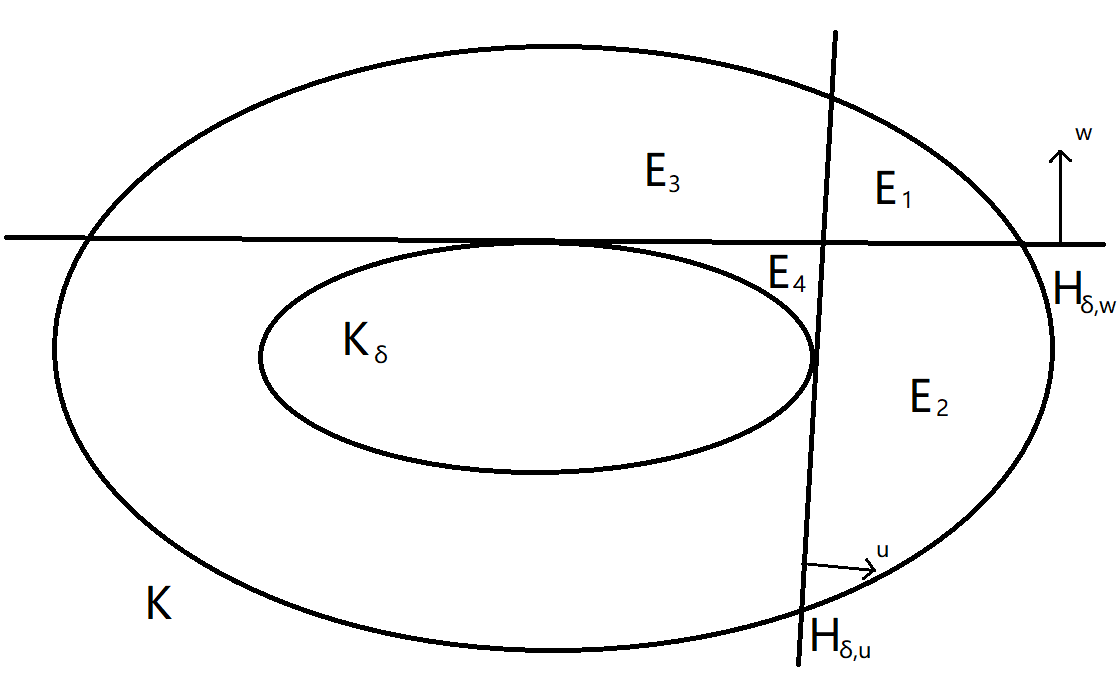}\\
	Figure 
\end{center}
Again, in order to compute the derivative of the boundary curve of $M_\delta(K)$ we
write
\begin{eqnarray*}
\delta  \   \text{vol}_2(K) \cdot  [ X_{K,\delta}(w) - X_{K,\delta}(u) ]
	&=& \int_{ E_1 \cup  E_3}x\,dx - \int_{ E_1 \cup  E_2}x\,dx \\
	&=& \int_{ E_3\cup  E_4}x\,dx - \int_ {E_2\cup  E_4}x\,dx . 
\end{eqnarray*}
Now we use the above introduced coordinate system. For  
 $n=n(\theta)$ and $T=T(\theta)$,  let $r_{K,K_\delta}(\theta)$ be such that $\gamma(\theta) + r_{K,K_\delta}(\theta) \  T \in \partial K$.  
As $K_\delta = K_{[\delta]}$,   $\gamma(\theta)$ is the midpoint of $n(\theta)^\perp \cap K$ by Dupin's characterization of $K_{[\delta]}$. Therefore 
\begin{eqnarray*}
&&\delta  \   \text{vol}_2(K) \cdot  [ X_{K,\delta}(w) - X_{K,\delta}(u) ] \\ &=& \int_{0}^{\phi}
  \int_{0}^{r_{K,K_\delta}(\theta)} 
	F(r,\theta)|r| \,dr\,d\theta -
  \int_{0}^{\phi}
  \int_{-r_{K,K_\delta}(\theta)}^{0}
	F(r,\theta)|r|\,dr\,d\theta  \\
&=&\int_{0}^{\phi}
  \int_{0}^{r_{K,K_\delta}(\theta)}
	F(r,\theta)r\,dr \,d\theta
+
  \int_{0}^{\phi}
  \int_{-r_{K,K_\delta}(\theta)}^{0}
	F(r,\theta)r\,dr \,d\theta  \\
&=&\int_{0}^{\phi}
\int_{-r_{K,K_\delta}(\theta)}^{r_{K,K_\delta}(\theta)}
F(r,\theta)r\,dr \,d\theta.
\end{eqnarray*}
Now  we  use  the definition of $F$. Thus
\begin{eqnarray*}
X_{K,\delta}(w) - X_{K,\delta}(u)
&=& \frac{1}{\delta  \   \text{vol}_2(K)} \int_{0}^{\phi}
\int_{-r_{K,K_\delta}(\theta)}^{r_{K,K_\delta}(\theta)} 
r \gamma(\theta) + r^2 T(\theta)\,dr \,d\theta  \\
&=& \frac{1}{\delta \  \text{vol}_2(K)} \int_{0}^{\phi}
\frac23  \  r_{K,K_\delta}^3(\theta) \  T(\theta) \,d\theta.
\end{eqnarray*}
As $M_\delta (K)$ is strictly convex and $C^1$ by \cite{HSW}, we can 
differentiate with respect to $\phi$ 
and get, 
\begin{equation}\label{eq-MdeltaK-gauss-deriv}
\frac{d X_{K,\delta}(n(\phi))}{d\phi} =
\frac{2\   r_{K,K_\delta}^3(\phi)}{3 \  \delta  \   \text{vol}_2(K)}  \  T(\phi).
\end{equation}
\par
\noindent
On the other hand, for any $\theta \in [0,2\pi]$,
\begin{eqnarray}\label{eq-section-2nd-moment}
&&\int_{K\cap H_{\delta, n(\theta)}} \langle x, T(\theta) \rangle^2 \,dx  = 
\int _{-r_{K,K_\delta}( \theta )}
^{ r_{K,K_\delta}( \theta )}
\langle  \gamma (\theta) +r T(\theta), T(\theta) \rangle^2 \   dr \nonumber\\
&&= 2 r_{K,K_\delta}(\theta ) \  \langle  \gamma (\theta), T(\theta) \rangle^2 +  2  
 \langle  \gamma (\theta), T(\theta) \rangle \int _{-r_{K,K_\delta}( \theta )}
^{ r_{K,K_\delta}( \theta )}
r  dr \nonumber +  \int _{-r_{K,K_\delta}( \theta )}
^{ r_{K,K_\delta}( \theta )} r^2 dr\\
&& =  \langle g(K \cap  H_{\delta, n(\theta)}), T(\theta) \rangle ^2  \  \text{vol}_1(K \cap  H_{\delta, n(\theta)} )  + 
\frac{2\   r_{K,K_\delta}^3(\theta)}{3},
\end{eqnarray}
since  $\gamma (\theta)$  is the centroid  $g(K \cap  H_{\delta, n(\theta)})$ of $K \cap  H_{\delta, n(\theta)}$.
Combining \eqref{eq-MdeltaK-gauss-deriv} and \eqref{eq-section-2nd-moment}, 
we get that for $\theta = \phi$,
\begin{equation*} \label{eq-isotropic1}
\left|\frac{d X_{K,\delta}(n(\phi))}{d\phi}  \right|= \frac{1}{\delta  \  \text{vol}_2(K)} \   \int_{K\cap H_{\delta, n(\phi)}} \langle x, T(\phi) \rangle^2 -  
\langle  g(K \cap  H_{\delta, n(\phi)}), T(\phi) \rangle^2 \,dx.
\end{equation*}
By \cite{HS, HSW}, the normal to $\partial M_\delta (K)$ at $X_{K,\delta} (n(\phi))$ is $n(\phi) = (\cos \phi, \sin \phi)$ and as $M_\delta (K)$ is strictly convex and $C^1$ by \cite{HSW}, 
$\xi(\phi) =
X_{K,\delta} (n(\phi))$ is a parametrization  of $\partial M_\delta (K)$ with
respect to the angle of the normal. The curvature is given by $\frac{d\phi}{ds}$
where $s$ is the arc length along the curve. Since $\frac{d\xi}{ds}$ is a unit
vector, we get by the chain rule $\frac{d\xi}{ds}=
\frac{d\xi}{d\phi} \frac{d\phi}{ds}$ that the radius of curvature is given by
\[
R(\phi) = \left|\frac{d X_{K,\delta}(n(\phi))}{d\phi}  \right| =
\frac{1}{\delta  \   \text{vol}_2(K)} \int_{K\cap H_{\delta, n(\phi)}} \langle x, T(\phi) \rangle^2  -  \langle  g(K \cap  H_{\delta, n(\phi)}), T(\phi) \rangle^2 \,dx.
\]
Since $M_\delta (K)$ is a disk if and only if its radius of curvature is constant,
the theorem follows.
\vskip 3mm
\noindent
Let now $n\geq 3$.
\par
\noindent
Let   $u \in S^{n-1}$ be arbitrary, but fixed, and let  $v \in S^{n-1}\cap u^\perp$.  We denote by  $W=\mbox{span}\{u,v\}$ the span of $u$ and $v$ and by  $W^\perp$  the  $(n-2)$-dimensional subspace that is the orthogonal complement of $W$.  $\bar K=K|W$ is the orthogonal projection of the convex body $K$ onto the $2$-dimensional subspace $W$. 
For a small  $\phi$, let $w = \cos \phi \  u + \sin \phi \  v.$ 
\newline
We define $\bar E_1$, $\bar E_2$ and  $\bar E_3$ as follows,  
$$
\bar E_1 =(H_{\delta, u}^+ \cap  H_{\delta,w}^+) \big | W, \  \   \bar E_2 =(H_{\delta, u}^+\cap H_{\delta,\eta}^-)\big |W, \  \  \bar E_3 = (H_{\delta, u}^-\cap H_{\delta, w}^+)\big |W
$$
and $\bar E_4$ is the curvilinear triangle enclosed by $H_{\delta, u}|W$, $H_{\delta, w} \big|W$, and the boundary of $\bar K_\delta=K_\delta|W$.
Then  the picture is identical to the previous Figure.
We let  
$$
E_i = \bar E_i\times W^\perp,   \ \   \text{ for} \  \   i=1, 2, 3,4.
$$
When  $K_\delta$ reduces to a point, we can assume without loss of generality that $K_\delta =\{0\}$. As noted before, the condition then reduces to 
$\int_{K \cap H_{\delta,u}}  \langle x, v \rangle^2 dx = C$.
In this case 
$\bar E_4 = \emptyset$ and the proof continues along the same lines 
as below.
Alternatively, one can  replace the coordinate system (\ref{def-coord-sys}) by the usual polar coordinates in $W$ as it was done in the case $n=2$. 
\par
\noindent
In the general case we thus have that 
\begin{eqnarray*}
&&\delta \  \vol(K) \left [X_{K,\delta}(w)-X_{K,\delta}(u)\right]
=\int_{K\cap H_{\delta, w}^+}x\,dx-\int_{K\cap H_{\delta, u}^+}x\,dx \\
&&=\int_{K\cap ( E_1\cup  E_3)}x\,dx-\int_{K\cap ( E_1\cup  E_2)}x\,dx
= \int_{K\cap  E_3}x\,dx-\int_{K\cap  E_2}x\,dx \\
&&= \int_{K\cap  (E_3 \cup E_4) }x\,dx-\int_{K\cap ( E_2\cup E_4) }x\,dx.
\end{eqnarray*}
For $x\in  W$, we consider the following parallel section function,
 \begin{equation} \label{AKW}
 A_{K,W} (x) = \mbox{vol}_{n-2} \left(K\cap \{W^\perp+x\}\right)
\end{equation}
and observe that by Fubini, 
\begin{equation*}
\delta \ \vol(K)  \left [ X_{K,\delta}(w) \right ]  = \int_{K \cap H_{\delta, w}^+}  z\   dz = \int_{\bar K} \left( \int_{(x+ W^\perp) \cap K \cap H_{\delta, w}^+ } y \  dy \right) dx.
\end{equation*}
We denote by $g(B) = \frac{1}{\vol(B)} \int_B y \  dy$ the centroid of the set $B$. Then we get
\begin{eqnarray*} \label{}
\delta(X_{K,\delta}(w))\big|W &= & \left(\int_{K\cap (E_3\cup  E_4)}x\,dx \right)  \Bigg | W  = 
\left(\int_{\bar K\cap (\bar E_3 \cup  \bar E_4)} \left( \int_{(x+W ^\perp) \cap K} y \ dy \right)  \,dx \right)\Bigg |  W \\
&=& \left(\int_{\bar K\cap (\bar E_3 \cup  \bar E_4)} \  A_{K,W} (x)\   g((x+W^\perp) \cap K) \  dx \right)\Bigg |  W\\
&=& \int_{\bar K\cap (\bar E_3 \cup  \bar E_4)} \  A_{K,W} (x)\  \left( g((x+W^\perp) \cap K)  \right)\big |  W \  dx \\
&=& \int_{\bar K\cap (\bar E_3 \cup  \bar E_4)} \  A_{K,W} (x)\  x \    dx,
\end{eqnarray*}
and similarly for $\delta \  \vol(K)(X_{K,\delta}(u))\big| W$.
Now we will use the  coordinate system  $F: \R\times [0, 2\pi] \to \R \setminus \text{int}\left(\bar K_\delta\right)$,  introduced earlier in (\ref{def-coord-sys}), 
\begin{equation*}
F(r,\theta) = \gamma(\theta) + r T(\theta).
\end{equation*}
We can assume that $n(0)=u$. Then $T(0)=v$.
We recall that the Jacobian of $F$ is given by $|r|$. 
We   abbreviate $n=n(\theta)$ and $T=T(\theta)$. We
let $r_{\bar K,\bar K_\delta}(n(\theta)), T(\theta)) = r_{\bar K,\bar K_\delta}(n, T)  >0$ be such that $\gamma(\theta) + r_{\bar K,\bar K_\delta}(n,T) \  T(\theta)  \in \partial \bar K$, 
and  $ r_{\bar K,\bar K_\delta}(n, -T)  >0$ be such that $\gamma(\theta) + r_{\bar K,\bar K_\delta}(n,T) \  (-T(\theta) ) \in \partial \bar K$.
We get 
\begin{align*}
&\hskip -15mm \delta  \  \vol(K) \left(X_{K,\delta}(w) -X_{K,\delta}(u)\right ) \big |W  \\
&\hskip -10mm =  \int_{0}^{\phi} \int_{0}^{r_{\bar K,\bar K_\delta}(n, T )} F (r,\theta) \   A_{K,W} (F(r,\theta))  \  |r|\,dr\,d\theta \\
&\hskip 25mm -
\int_{0}^{\phi} \int_{- r_{\bar K,\bar K_\delta}(n, -T )}^{0}  F(r,\theta) \  A_{K,W} (F(r,\theta))  \  |r|\,dr
\,d\theta \\
& \hskip -10mm =\int_{0}^{\phi}\int_{0}^{r_{\bar K,\bar K_\delta}(n, T )}   F (r,w) \   A_{K,W} (F(r,\theta))  \   r\,dr \,d\theta \\
&\hskip 25mm + \int_{0}^{\phi}\int_{-r_{\bar K,\bar K_\delta}(n, -T)}^{0}  F (r,\theta) \   A_{K,W} (F(r,\theta))  \   r\,dr \,d\theta \\
&\hskip -10mm =  \int_{0}^{\phi}\int_{-r_{\bar K,\bar K_\delta}(n, -T )}^{r_{\bar K,\bar K_\delta}(n, T)}  F (r,\theta) \   A_{K,W} (F(r,\theta))  \   r\,dr \,d\theta .
\end{align*}
We differentiate  with respect to $\phi$,   
\begin{eqnarray*}\label{equ2}
&& \delta \   \vol(K) \frac{d}{d\phi}\left((X_{K,\delta}(w)-X_{K,\delta}(u))\big|W\right) = \delta \   \vol(K) \frac{d}{d\phi}\left( (X_{K,\delta}(w))\big|W\right) \\
&&=\int_{-r_{\bar K,\bar K_\delta}(n(\phi), -T(\phi) )}^{r_{\bar K,\bar K_\delta}(n(\phi),T(\phi))} F(r, \phi) A_{K,W} (F(r, \phi)) \  r\,dr.
\end{eqnarray*}
Putting  $\phi=0$, results in
\begin{eqnarray}\label{GL1}
&& \delta \   \vol(K) \frac{d}{d\phi}\left( (X_{K,\delta}(w))\big|W\right) \Bigg| _{\phi = 0} 
=  \int_{-r_{\bar K,\bar K_\delta}(u, -v )}^{r_{\bar K,\bar K_\delta}(u,v)} F(r, 0) A_{K,W} (F(r, 0)) \  r\,dr \nonumber \\
&&= \int_{-r_{\bar K,\bar K_\delta}(u, -v )}^{r_{\bar K,\bar K_\delta}(u,v)} [\gamma(0) +r v] \  A_{K,W} (F(r, 0)) \  r\,dr \nonumber \\ 
&&=  \int_{-r_{\bar K,\bar K_\delta}(u, -v )}^{r_{\bar K,\bar K_\delta}(u,v)} \gamma(0) \  A_{K,W} (F(r, 0)) \  r\,dr  +   \int_{-r_{\bar K,\bar K_\delta}(u, -v )}^{r_{\bar K,\bar K_\delta}(u,v)} r^2 \  v \  A_{K,W} (F(r, 0)) \,dr  \nonumber\\ 
&&= v \   \int_{-r_{\bar K,\bar K_\delta}(u, -v )}^{r_{\bar K,\bar K_\delta}(u,v)} r^2 \    A_{K,W} (F(r, 0)) \,dr,
\end{eqnarray}
where the  last equality holds by Dupin since $H_{\delta, u}\cap K_\delta$ is the centroid of $H_{\delta, u}\cap K$, i.e.
\begin{equation}\label{null}
 \int_{-r_{\bar K,\bar K_\delta}(u, -v )}^{r_{\bar K,\bar K_\delta}(u,v)}  \gamma(0)  \  A_{K,W} (F(r, 0)) \  r\,dr = 0.
\end{equation}
Indeed, in the coordinate system (\ref{def-coord-sys}), the centroid of $H_{\delta, u}\cap K$ is $\gamma (0)$. Thus, with 
coordinate system (\ref{def-coord-sys}) we get as above
\begin{eqnarray*}
&&\text{vol}_{n-1} \left(K\cap H_{\delta, u} \right)  \langle v,  \gamma (0) \rangle  =   \int _{K\cap H_{\delta, u}} \langle v,  x \rangle \   dx \\
&&= \int_{-r_{\bar K,\bar K_\delta}(u, -v )}^{r_{\bar K,\bar K_\delta}(u,v)} \langle \gamma(0) +r v, v\rangle \  A_{K,W} (F(r, 0)) \  dr \\
&&= \langle \gamma(0), v \rangle \int_{-r_{\bar K,\bar K_\delta}(u, -v )}^{r_{\bar K,\bar K_\delta}(u,v)}   A_{K,W} (F(r, 0)) \  dr
+  \langle  v, v\rangle   \int_{-r_{\bar K,\bar K_\delta}(u, -v )}^{r_{\bar K,\bar K_\delta}(u,v)} r \  A_{K,W} (F(r, 0)) \  dr \\
&&= \langle \gamma(0), v \rangle  \  \text{vol}_{n-1} \left(K\cap H_{\delta, u} \right) + \int_{-r_{\bar K,\bar K_\delta}(u, -v )}^{r_{\bar K,\bar K_\delta}(u,v)} r \  A_{K,W} (F(r, 0)) \  dr.
\end{eqnarray*}
\par
\noindent
On the other hand,  again in the coordinate system (\ref{def-coord-sys}), and also using (\ref{null}), 
\begin{eqnarray} \label {GL2}
&&\int_{K\cap H_{\delta, u}} \langle x, v \rangle^2\ dx  = \int _{-{r_{\bar K,\bar K_\delta}(u,-v)}} ^{r_{\bar K,\bar K_\delta}(u,v)}  \langle \gamma(0) +r v, v\rangle ^2\   A_{K,W}(F(r,0))  dr \nonumber \\
&&= \langle \gamma(0), v\rangle  ^2 \int _{-{r_{\bar K,\bar K_\delta}(u,-v)}} ^{r_{\bar K,\bar K_\delta}(u,v)}   A_{K,W}(F(r,0))  dr + \int _{-{r_{\bar K,\bar K_\delta}(u,-v)}} ^{r_{\bar K,\bar K_\delta}(u,v)} r^2 \    A_{K,W}(F(r,0)) dr  \nonumber \\
&&= \langle \gamma(0), v\rangle  ^2  \ \text{vol}_{n-1} \left(K\cap H_{\delta, u} \right) + \int _{-{r_{\bar K,\bar K_\delta}(u,-v)}} ^{r_{\bar K,\bar K_\delta}(u,v)} r^2 \    A_{K,W}(F(r,0)) dr.
\end{eqnarray}
As $w=\cos \phi  \  u + \sin \phi \  v$,  it follows from (\ref{GL1}) and (\ref{GL2}) that
\begin{eqnarray*}
&&\hskip -4mm \Bigg |  \frac{d}{d\phi}\left( X_{K,\delta}(\cos \phi \  u + \sin \phi \  v)\Big| W\right) \Big| _{\phi = 0} \Bigg |= 
\frac{1}{ \delta   \vol(K)} \  \int_{K\cap H_{\delta, u}} \left(\langle x, v \rangle^2 - \langle \gamma(0), v\rangle  ^2 \right)dx \\
&&= \frac{1}{ \delta \   \vol(K)} \  \int_{K\cap H_{\delta, u}} \left(\langle x, v \rangle^2 - \langle g(K\cap H_{\delta, u}), v\rangle  ^2 \right)\ dx.
\end{eqnarray*}
Observe that in the case when $K_\delta =\{0\}$, $\langle g(K\cap H_{\delta, u}), v\rangle =0$. We have that 
$w= n(\phi)= \cos \phi  \  u + \sin \phi \  v  \in W$ is the outer unit normal to $M_\delta(K)$ in  $X_{\delta, K}(w)$. Therefore, 
$w$ is the  outer unit normal to $M_\delta(K)\big|W$ in $X_{\delta, K}(w)\big|W$. Again, as $M_\delta(K)$ and therefore $M_\delta(K)\big|W$
is strictly convex and $C^1$ by \cite{HSW}, 
$X_{\delta, K}(n(\phi))$ is a parametrization of the boundary of
$M_\delta(K)\big|W$ with respect to the angle of the normal.
Thus the curvature of $M_\delta(K) \big|W$ is constant, which implies that $M_\delta(K) \big|W$ is a disk. Since $W$ is  arbitrary, we get that every two dimensional projection of $M_\delta$ is a disk, and it follows that $M_\delta(K)$ is a Euclidean ball (\cite{GardnerBook}, Corollary 3.1.6).

\end{proof}

\vskip 3mm
\noindent
\subsection{Proof of Theorem \ref{thm2}}
\begin{proof}
Since $K$ is symmetric and has volume $1$ and  density $\rho=\frac{1}{2}$, we have that $\delta = \frac{1}{2}$, as noted above. Therefore,  $K_{\frac12} = K_{[\frac12]} = \{0\}$.
Since $K$ is an Ulam floating body,  the remark after Theorem \ref{thm1} implies that  for
any $u \in S^{n-1}$ and $v \in u^\perp\cap S^{n-1}$
	\begin{equation}\label{intD}
	\int_{u^\perp \cap K}\langle x, v\rangle^2\,dx =C, 
	\end{equation}
	for some constant $C$. 
Let $u \in S^{n-1}$ be arbitrary, but fixed.  We pass to  polar coordinates in $u^\perp$ and get for all $v \in u^\perp\cap S^{n-1}$, 
\begin{eqnarray*}
C= \int_{u^\perp \cap S^{n-1}}  \int_{t=0}^ {r_K(\xi)} t^n \langle \xi, v\rangle^2 dt \    d\sigma (\xi)  = \frac{1}{n+1}  \int_{u^\perp \cap S^{n-1}} r_K(\xi)^{n+1}  
\langle \xi, v\rangle^2   d\sigma (\xi).
\end{eqnarray*}
Now we integrate over all $v \in u^\perp\cap S^{n-1} = S^{n-2}$ w.r. to the normalized Haar measure $\mu$ on $S^{n-2}$.  We use that $\int_{S^{n-2} }   
\langle \xi, v\rangle^2  d\mu (v) = c_n \|\xi\| = c_n$, where $c_n= 2 \frac{\text{vol}_{n-2} \left(B^{n-2}\right)}{\text{vol}_{n-2} \left(S^{n-2}\right)}$ and  get that
\begin{eqnarray*}
\frac{(n+1) C}{c_n}=  \int_{u^\perp \cap S^{n-1}}  r_K(\xi)^{n+1}  d\sigma (\xi) = R \  r_K^{n+1}(u),
\end{eqnarray*}
where $R$ is the spherical Radon transform (\ref{Radon}).  We rewrite this equation as
\begin{eqnarray*}
\int_{u^\perp \cap S^{n-1}} d\sigma (\xi) =  \int_{u^\perp \cap S^{n-1}} \frac{2 \text{vol}_{n-2} \left(B^{n-2}\right)}{(n+1) C}  r_K(\xi)^{n+1}  d\sigma (\xi), 
\end{eqnarray*}
or
\begin{eqnarray*}
0 =  \int_{u^\perp \cap S^{n-1}} \left( \frac{2 \text{vol}_{n-1} \left(B^{n-2}\right)}{(n+1) C}  r_K(\xi)^{n+1} - 1 \right) d\sigma (\xi). 
\end{eqnarray*}
As $u$ was arbitrary and as $r_K$ is even, it then follows from e.g., Theorem C.2.4  of \cite{GardnerBook} that $r_K = \text{const.}$  for $\sigma$ almost all $u$
and as $r_K$ is continuous, $r_K = \text{const.}$ on $S^{n-1}$. Thus $K$ is a ball.
\end{proof}

\smallskip \noindent
Dan I. Florentin \\
Department of Mathematics\\
Bar-Ilan University, Israel \\
{\it e-mail}: danflorentin@gmail.com \\

\smallskip \noindent
Carsten Sch\"utt \\
Mathematisches Seminar \\
Christian-Albrechts University of Kiel\\
Ludewig-Meyn-Strasse 4, 24098 Kiel, Germany \\
{\it e-mail}: schuett@math.uni-kiel.de \\

\smallskip \noindent
Elisabeth M. Werner\\
 Department of Mathematics \ \ \ \ \ \ \ \ \ \ \ \ \ \ \ \ \ \ \ Universit\'{e} de Lille 1\\
Case Western Reserve University \ \ \ \ \ \ \ \ \ \ \ \ \ UFR de Math\'{e}matique \\
 Cleveland, Ohio 44106, U. S. A. \ \ \ \ \ \ \ \ \ \ \ \ \ \ \ 59655 Villeneuve d'Ascq, France\\
{\it elisabeth.werner@case.edu}\\ \\

\smallskip \noindent
Ning Zhang \\
School of Mathematics and Statistics \\
Huazhong University of Science and Technology\\
1037 Luoyu Road, Wuhan, Hubei 430074, China \\
{\it e-mail}: nzhang2@hust.edu.cn \\

\end{document}